\documentclass[a4paper, 12pt]{article}
\usepackage{mathrsfs}
\usepackage{amsmath}
\usepackage{amsfonts}
\usepackage[T1]{fontenc}
\usepackage[latin9]{inputenc}
\usepackage{amsthm}
\usepackage{graphicx}
\usepackage{amsmath}

\usepackage{amsthm}
\usepackage{array}
\usepackage{cases}
\makeatletter
\def\th@plain{%
  \upshape 
}
\makeatother

\makeatletter
\renewenvironment{proof}[1][\proofname]{\par
  \pushQED{\qed}%
  \normalfont \topsep6\p@\@plus6\p@\relax
  \trivlist
  \item[\hskip\labelsep
        \bfseries
    #1\@addpunct{.}]\ignorespaces
}{%
  \popQED\endtrivlist\@endpefalse
}
\makeatother

\numberwithin{equation}{section}

\newtheorem{thm}{Theorem}[section]

\newtheorem{lem}[thm]{Lemma}
\newtheorem{prop}{Proposition}

\newtheorem{pblm}[thm]{Problem}
\newtheorem{clm}{Claim}

\numberwithin{equation}{section}

\usepackage[varg]{txfonts}
\usepackage{graphicx, epsfig, subfigure}
\usepackage{enumerate} 
\usepackage[square, numbers, sort&compress]{natbib} 
\ifx\pdfoutput\undefined
 \usepackage[dvipdfm,%
  pdfstartview=FitH, 
  bookmarks=true,%
  bookmarksnumbered=true, 
  bookmarksopen=true, 
  plainpages=false,%
  pdfpagelabels,%
  colorlinks=true, 
  linkcolor=blue, 
  citecolor=blue,%
  urlcolor=black,
  pdfborder=001]{hyperref}
  \AtBeginDvi{}  
\else
 \usepackage[pdftex,%
  pdfstartview=FitH, 
  bookmarks=true,%
  bookmarksnumbered=true, 
  bookmarksopen=true, 
  plainpages=false,%
  pdfpagelabels,%
  colorlinks=true, 
  linkcolor=blue, 
  citecolor=blue,%
  urlcolor=black,
  pdfborder=001]{hyperref}
\fi

\numberwithin{equation}{section}

\newcommand{\gx}{G^{\times}}

\begin{document}

\title{\LARGE Light edges in 1-planar graphs of minimum degree 3
\thanks{Supported by the National Natural Science Foundation of China (11871055, 11301410) and the Youth Talent Support Plan of Xi'an Association for Science and Technology (2018-6).}}
\author{Bei Niu, \,\,Xin Zhang \thanks{Corresponding author. Email: xzhang@xidian.edu.cn.}\\
{\small  School of Mathematics and Statistics, Xidian University, Xi'an, Shaanxi, 710071, China}
}
\date{}

 \maketitle

\begin{abstract}\baselineskip 0.56cm
A graph is 1-planar if it can be drawn in the plane so that each edge is crossed by at most one another edge. In this work we prove that each 1-planar graph of minimum degree at least $3$ contains an edge with degrees of its endvertices of type  $(3,\leq23)$ or $(4,\leq11)$ or $(5,\leq9)$ or $(6,\leq8)$ or $(7,7)$. Moreover, the upper bounds $9,8$ and $7$ here are sharp and the upper bounds $23$ and $11$ are very close to the possible sharp ones, which may be 20 and 10, respectively. This generalizes a result of Fabrici and Madaras [\emph{Discrete Math.}, 307 (2007) 854--865] which says that each 3-connected 1-planar graph contains a light edge, and improves a result of Hud\'ak and  \v Sugerek [\emph{Discuss.\,Math.\,Graph Theory}, 32(3) (2012) 545--556], which states that each 1-planar graph of minimum degree at least $4$ contains an edge with degrees of its endvertices of type $(4,\leq 13)$ or $(5,\leq 9)$ or $(6,\leq 8)$ or $(7, 7)$.
\vspace{3mm}

\end{abstract}

\baselineskip 0.56cm

\section{Introduction }

All graphs considered in this paper are finite, simple and undirected. Notations are standard (cf.\,\cite{Bondy}) unless we state otherwise.

A \emph{planar graph} is a graph that can be drawn in the plane in such a way that no edges cross each other. Such a drawing is called a \emph{plane graph}. For a plane graph $G$, $V(G),E(G)$ and $F(G)$ denote the set of vertices, edges, and faces of $G$, respectively.
A $k$-, $k^+$- and $k^-$\emph{-vertex} (resp.\,\emph{face}) is a vertex (resp.\,face) of degree $k$, at least $k$ and at most $k$, respectively.
An edge $uv$ is \emph{of type $(a,\leq b)$} if $d(u)=a$ and $d(v)\leq b$. Similarly we can define edges of type $(\leq a, \leq b)$ or $(a,\geq b)$ or $(\geq a, \geq b)$.
A graph is \emph{1-planar} if it can be drawn in the plane so that each edge is crossed by at most one another edge. Such a drawing so that the number of crossings is as small as possible is called a \emph{1-plane graph}.
The notion of 1-planarity was introduced by Ringel \cite{R65} while trying to simultaneously color the vertices and faces of a plane graph such that any pair of adjacent or incident elements receive different colors.

A well-known consequence of the Euler's Polyhedron Formula says that each planar graph has a vertex of degree at most 5. The beautiful
Kotzig's Theorem \cite{K.1955} states that each 3-connected planar graph contains an edge whose sum of degrees of
its endvertices is at most 13, and at most 11 if 3-vertices are absent. In addition, the bounds 13 and 11 are sharp. For other relative results on the light subgraphs of graphs embedded in the plane, we refer the readers to a recent survey contributed by Jendrol' and Voss \cite{JV}.

For 1-planar graphs, there are analogical results. For example, Fabrici and Madaras \cite{FM07} showed that each 1-planar graph contains a vertex of degree at most 7, and proved that each 3-connected 1-planar graph contains an edge with both endvertices of degrees at most
20. Here the bound 20 is also sharp.

As we know, every 3-connected graph has minimum degree at least 3. Hence a natural question is to ask whether each 1-planar graph of minimum degree at least 3 contains a light edge (i.e., an edge such that the sum, or the maximum, of degrees of its endvertices is bounded by a constant that is independent of the given graph). Actually, the answer to the above question is positive for 1-planar graph of minimum degree at least 4. Precisely, Hud\'ak and  \v Sugerek  \cite{HS12} proved

\begin{thm}\textcolor[rgb]{1.00,0.00,0.00}{\cite{HS12}}\label{HS-thm}
Each 1-planar graph of minimum degree at least $4$ contains an edge of type $(4,\leq 13)$ or $(5,\leq 9)$ or $(6,\leq 8)$ or $(7, 7)$.
\end{thm}

\noindent Moreover, they also claimed that for 1-planar graphs of minimum degree at least $5$, these bounds in Theorem \ref{HS-thm} are best possible and the list of edges is minimal (in the
sense that, for each of the considered edge types there are 1-planar graphs whose set of types of edges contains just the selected edge type). Actually, there exists 1-planar graph with only edges of type $(5,9)$, $(5,10)$ and $(9,10)$, or with only edges of type $(6,8)$ and $(8,8)$, or with only edges of type $(7,7)$. The first two graphs were constructed by Hud\'ak and \v Sugerek \cite{HS12}, and the last graph (i.e., 7-regular 1-planar graph) was introduced by Fabrici and Madaras \cite{FM07}.

Motivated by Theorem \ref{HS-thm} of Hud\'ak and \v Sugerek, and also by the above mentioned result of Fabrici and Madaras \cite{FM07}, we investigate light edges in 1-planar graphs by proving that each 1-planar graph of minimum degree at least 3 contains a light edge. More precisely, we are able to prove the following main theorem of this paper.

\begin{thm}\label{our-thm}
Each 1-planar graph of minimum degree at least $3$ contains an edge of type $(3,\leq23)$ or $(4,\leq11)$ or $(5,\leq9)$ or $(6,\leq8)$ or $(7,7)$.
\end{thm}

Clearly, Theorem \ref{our-thm} can be seen as an improvement and also a generalization of Theorem \ref{HS-thm}.
Although we improve 13 in Theorem \ref{HS-thm} to 11, we still do not know whether 11 is sharp. If it can be improved, then it shall be 10, since
Hud\'ak and \v Sugerek \cite{HS12} constructed a 1-planar graph with only edges of type $(4,10)$ and $(10,10)$. On the other hand, the sharpness of the upper bound 23 in Theorem \ref{our-thm} is unclear. Since Fabrici and Madaras \cite{FM07} constructed a 1-planar graph with only edges of type $(3,20)$ and $(20,20)$, we want to know whether the upper bound 23 in Theorem \ref{our-thm} can be replaced by 20. In conclusion, we raise the following problem.

\begin{pblm}
Does each 1-planar graph of minimum degree at least $3$ contain an edge of type $(3,\leq20)$ or $(4,\leq10)$ or $(5,\leq9)$ or $(6,\leq8)$ or $(7,7)$?
\end{pblm}
%

\section{The existence of a light edge}

The \emph{associated plane graph} $\gx$ of a 1-plane graph $G$ is the plane graph that is obtained from $G$ by turning
all crossings of $G$ into new vertices of degree four. Those new 4-vertices are called \emph{false vertices} of $\gx$, and the original vertices of $G$ are called \emph{true vertices} of $\gx$. A face of $\gx$ is \emph{false} if it is incident with at least one false vertex, and \emph{true} otherwise.

\begin{lem}\label{n.adj}
If $G$ is a $1$-plane graph, then

(a) false vertices in $\gx$ are not adjacent;

(b) if a $3$-vertex $v$ is incident with two
$3$-faces and adjacent to two false vertices in $\gx$, then $v$ is incident with a $5^+$-face;

(c) there exists no edge $uv$ in $\gx$ such that $d_{\gx}(v)=3$, $u$ is a false vertex, and $uv$ is incident with two $3$-faces;

(d) if $v$ is a true 4-vertex in $\gx$, then $v$ is incident with at most three false 3-faces.
\end{lem}

\begin{proof}
The conclusions (a), (b) and (c) come from \cite[Lemma 1]{ZW11}. For (d), suppose that $v$ is a true 4-vertex in $\gx$ incident with four false 3-faces $\{vv_1v_2\}, \{vv_2v_3\}, \{vv_3v_4\}$ and $\{vv_4v_1\}$. By (a), we assume, without loss of generality, that $v_1$ and $v_3$ are false. Now, there are two edges in $G$ connecting $v_2$ to $v_4$, one of which passes through $v_1$ and the other passes through $v_3$. This contradicts the fact that $G$ is simple.
\end{proof}

\begin{proof}[\emph{The Proof of Theorem \ref{our-thm}}]
Suppose that there is a 1-plane graph $G$ of minimum degree at least 3 contradicting Theorem \ref{our-thm}. So $G$ contains only edges of type $(3,\geq24)$ or $(4,\geq12)$ or $(5,\geq10)$ or $(6,\geq9)$ or $(\geq7,\geq8)$. We apply the discharging method to the associated plane graph $\gx$ of $G$. Formally, for each vertex $v\in V(\gx)$, let $c(v):=d_{\gx}(v)-4$ be its initial charge, and for each face $f\in F(\gx)$, let $c(f):=d_{\gx}(f)-4$ be its initial charge. Clearly,
$$\sum\limits_{x\in V(\gx)\bigcup F(\gx)}c(x)=-8<0$$ by the well-known Euler's formula.

Let $f=\{vv_{1}v_{2}\}$ be a false 3-face of ${\gx}$ such that $v$ is a false vertex generating by $v_{1}v_{3}$ crossing $v_{2}v_{4}$ in $G$.
If $d_{\gx}(v_{1})=k,~v_{2}v_{3}\in E(G)$ and
$$ d_{\gx}(v_{4})\leq\left\{
\begin{aligned}
11&,~~{\rm if}~~k=4 \\
9&,~~{\rm if}~~k=5 \\
8&,~~{\rm if}~~k=6
\end{aligned}
\right.
$$
then $f$ is said to be \emph{k-special}, where $4\leq k\leq6$.


We define discharging rules as follows.

\begin{description}\vspace{-.8em}
  \item[R1] Every true 4-vertex of $\gx$ sends $\frac{1}{6}$ to each of its incident 4-special faces.
  \item[R2] Every 5-vertex of $\gx$ sends $\frac{3}{10}$ to each of its incident 5-special faces, and $\frac{1}{5}$ to each of its incident 3-faces that are not 5-special.
  \item[R3] Every 6-vertex of $\gx$ sends $\frac{7}{18}$ to each of its incident 6-special faces, and $\frac{1}{3}$ to each of its incident 3-faces that are not 6-special.
  \item[R4] Every 7-vertex of $\gx$ sends $\frac{1}{2}$ to each of its incident false 3-faces.
  \item[R5] Every $8^{+}$-vertex $v$ of $\gx$ sends $\frac{d_{\gx}(v)-4}{d_{\gx}(v)}$ to each of its incident faces.
  \item[R6] Let $v$ be a false vertex of $\gx$ such that $v_{1}v_{3}$ crossed $v_{2}v_{4}$ in $G$ at $v$, and let $f_{i}$ with $1\leq i \leq4$ be the face that is incident with $vv_{i}$ and $vv_{i+1}$ in $\gx$ (here $v_{5}$ is recognized as $v_{1}$).
        \begin{description}
        \item[R6.1] Suppose that min$\{d_{\gx}(v_{1}),d_{\gx}(v_{2})\}\geq24$ and $d_{\gx}(v_{3})=3$.
            \item[$\bullet$]  If $d_{\gx}(v_{4})=3$, then $f_{1}$ sends $\frac{1}{6}$, through $v$, to each of the elements among $f_{2}$, $f_{4}$, $v_{3}$, $v_{4}$.
            \item[$\bullet$] If $d_{\gx}(v_{4})\geq4$, then $f_{1}$ sends $\frac{1}{3}$ to both $f_{2}$ and $v_{3}$ through $v$.
        \item[R6.2] Suppose that $23\geq $min$\{d_{\gx}(v_{1}),d_{\gx}(v_{2})\}\geq12$ and $d_{\gx}(v_{3})\leq6.$
            \begin{description}
            \item[$\bullet$] If $f_{1}$ is a 3-face, then $f_{1}$ sends $\frac{1}{6}$ to both $f_{2}$ and $f_{4}$ through $v$ while $d_{\gx}(v_{4})\leq6$, and $\frac{1}{3}$ to $f_{2}$ through $v$ while $d_{\gx}(v_{4})\geq7$.
            \item[$\bullet$] If $f_{1}$ is a $4^{+}$-face, then $f_{1}$ sends $\frac{1}{3}$ to both $f_{2}$ and $f_{4}$ through $v$.
            \end{description}
        \item[R6.3] Suppose that $11\geq$ min$\{d_{\gx}(v_{1}),d_{\gx}(v_{2})\}\geq10$ and $d_{\gx}(v_{3})\leq6.$
            \begin{description}
            \item[$\bullet$] If $f_{1}$ is a 3-face, then $f_{1}$ sends $\frac{1}{10}$ to both $f_{2}$ and $f_{4}$ through $v$ while $d_{\gx}(v_{4})\leq6$, and $\frac{1}{5}$ to $f_{2}$ through $v$ while $d_{\gx}(v_{4})\geq7$.
            \item[$\bullet$] If $f_{1}$ is a $4^{+}$-face, then $f_{1}$ sends $\frac{3}{10}$ to both $f_{2}$ and $f_{4}$ through $v$.
            \end{description}
        \item[R6.4] Suppose that min$\{d_{\gx}(v_{1}),d_{\gx}(v_{2})\}=9$ and $d_{\gx}(v_{3})\leq6.$
            \begin{description}
            \item[$\bullet$] If $f_{1}$ is a 3-face, then $f_{1}$ sends $\frac{1}{18}$ to both $f_{2}$ and $f_{4}$ through $v$ while $d_{\gx}(v_{4})\leq6$, and $\frac{1}{9}$ to $f_{2}$ through $v$ while $d_{\gx}(v_{4})\geq7$.
            \item[$\bullet$] If $f_{1}$ is a $4^{+}$-face, then $f_{1}$ sends $\frac{5}{18}$ to both $f_{2}$ and $f_{4}$ through $v$.
            \end{description}
        \end{description}
  \item[R7] Every $4^-$-face of $\gx$ redistributes its remaining charge after applying the previous rules equitably to each of its incident true $4^{-}$vertices. \vspace{-.8em}
  \item[R8] Every $5^{+}$-face of $\gx$ sends $\frac{2}{3}$ to each of its incident 3-vertices, and then redistributes its remaining charge after applying the previous rules equitably to each of its incident true 4-vertices.
\end{description}

Let $c'(x)$ be the charge of $x\in V(\gx)\cup F(\gx)$ after applying the above rules. Since our rules only move charge around, and do not affect
the sum, we have $$\sum_{x\in V(\gx)\cup F(\gx)}c'(x)=\sum_{x\in V(\gx)\cup F(\gx)}c(x)<0.$$
Next, we prove
that $c'(x)\geq 0$ for each $x\in V(\gx)\cup F(\gx)$. This leads to $$\sum\limits_{x\in V(\gx)\cup F(\gx)}c'(x)\geq 0,$$ a contradiction.

\begin{clm}\label{clm1}
Every true 3-face incident with one 3-vertex $v$ sends at least $\frac{2}{3}$ to $v$.
\end{clm}

\begin{proof}
Such a true 3-face sends to $v$ at least $2\times \frac{24-4}{24}-1=\frac{2}{3}$ by R5 and R7, since the neighbors of $v$ on this face are $24^+$-vertices.
\end{proof}

\begin{clm}\label{clm2}
Every true 3-face incident with one 4-vertex $v$ sends at least $\frac{1}{3}$ to $v$.
\end{clm}

\begin{proof}
Such a true 3-face sends to $v$ at least $2\times \frac{12-4}{12}-1=\frac{1}{3}$ by R5 and R7, since the neighbors of $v$ on this face are $12^+$-vertices.
\end{proof}

A \emph{transitive false vertex} $v$ on $f\in F(\gx)$ is a false vertex such that its two neighbors $u,w$ on $f$ have degrees both at least 9. If $f$ sends out charges via a false vertex, then this false vertex must be transitive by R6.

\begin{clm}\label{clm-new}
Let $f$ be a face in $\gx$ and let $\rho^+(f),\rho^-(f)$ respectively be the total charges that $f$ receives from its incident $9^+$-vertices, and that $f$ sends out via its incident transitive false vertices. If $d_{\gx}(f)\geq 4$, then $\rho^+(f)\geq \rho^-(f)$, and if $d_{\gx}(f)=3$, then $\rho^+(f)\geq \rho^-(f)+1$.
\end{clm}

\begin{proof}
If $f$ is true, or not incident with a transitive false vertex, then there is nothing to prove. Hence we assume that there are some transitive false vertices $v_1,v_2,\ldots,v_k$ on $f$. For each $v_i$ with $1\leq i\leq k$, let $u_i$ and $w_i$ be two neighbors of $v_i$ on the face $f$. Since false vertices are not adjacent in $\gx$, $u_i$ and $w_i$ are $9^+$-vertices by the definition of the transitive false vertex.
The \emph{contribution} of $v_i$ ($1\leq i\leq k$) is denoted by $\pi^+(v_i)=\frac{d_{\gx}(u_i)-4}{d_{\gx}(u_i)}+\frac{d_{\gx}(w_i)-4}{d_{\gx}(w_i)}$, and the \emph{demand} $\pi^-(v_i)$ of $v_i$ is the amount of charges that $f$ sends out via $v_i$. By R6, one can check that $\pi^+(v_i)\geq 2\pi^-(v_i)$ for each $1\leq i\leq k$. Therefore, $$\rho^+(f)\geq \frac{1}{2}\sum\limits_{i=1}^k \pi^+(v_i)\geq \frac{1}{2}\sum\limits_{i=1}^k 2\pi^-(v_i)=\sum\limits_{i=1}^k \pi^-(v_i)=\rho^-(f)$$ if $d_{\gx}(f)\geq 4$. On the other hand, if $d_{\gx}(f)=3$, then it is easy to see from R6 that
 $\pi^+(v_1)-\pi^-(v_1)\geq \min\{2\times \frac{5}{6}-4\times\frac{1}{6},2\times \frac{2}{3}-2\times\frac{1}{6},2\times \frac{3}{5}-2\times\frac{1}{10},2\times \frac{5}{9}-2\times\frac{1}{18}\}=1$, which implies that $\rho^+(f)=\pi^+(v_1)\geq \pi^-(v_1)+1=\rho^-(f)+1$.
\end{proof}


\begin{clm}\label{new-2}
Suppose that $f$ is a 4-face that is not incident with two false vertices.

(1) If $f$ is incident with at least one 3-vertex, then $f$ sends at least $\frac{5}{12}$ to each of its incident true $4^-$-vertices;

(2) If $f$ is not incident with any 3-vertex and $f$ is incident with at least one true 4-vertex, then  $f$ sends at least $\frac{1}{3}$ to each of its incident true $4$-vertices.
\end{clm}

\begin{proof}
(1) Let $f=\{v_1v_2v_3v_4\}$ be such a 4-face so that $d_{\gx}(v_1)=3$. Since $v_2$ and $v_4$ cannot be both false, at least one of them is a $24^+$-vertex, and moreover, neither $v_2$ nor $v_4$ can be a transitive false vertex.
If $v_3$ is true (so may be a true $4^-$-vertex), then $f$ sends to each of its incident true $4^-$-vertices at least $\frac{1}{2}\times (4-4+\frac{5}{6})=\frac{5}{12}$ by R5 and R7.
If $v_3$ is false, then it is a transitive false vertex because $v_2$ and $v_4$ are $24^+$-vertices.
So $f$ sends to each of its incident true $4^-$-vertices at least $4-4+2\times\frac{5}{6}-2\times\frac{1}{3}=1$ by R5, R6.1 and R7.

(2) Let $f=\{v_1v_2v_3v_4\}$ be such a 4-face so that $v_1$ is a true 4-vertex. Since $v_2$ and $v_4$ cannot be both false, at least one of them is a $12^+$-vertex, and moreover, neither $v_2$ nor $v_4$ can be a transitive false vertex.

If $v_3$ is true (so may be a true $4$-vertex), then $f$ sends to each of its incident true $4$-vertices at least $\frac{1}{2}\times (4-4+2\times\frac{1}{3})=\frac{1}{3}$ by R5 and R7.
If $v_3$ is false, then it is a transitive false vertex because $v_2$ and $v_4$ are $12^+$-vertices.
So $f$ sends to each of its incident true $4^-$-vertices at least $4-4+2\times\frac{2}{3}-\textcolor[rgb]{1.00,0.00,0.00}{2\times\frac{1}{3}}=\frac{2}{3}$ by R5, R6.1, R6.2 and R7.
\end{proof}

\begin{clm}\label{clm3}
Every $5^{+}$-face incident with true $4^{-}$-vertices sends at least $\frac{1}{3}$ to each of its incident true 4-vertices (if exist).
\end{clm}

\begin{proof}
Suppose that $f$ is incident with $s$ 3-vertices, and $t$ true 4-vertices.
If $t=0$, then there is noting to be proved, so we may assume that $t\geq 1$.
Since true $4^-$-vertices are not adjacent in $\gx$, $s+t\leq\lfloor\frac{d_{\gx}(f)}{2}\rfloor$. By R8 and Claim \ref{clm-new},
$f$ sends to each of its incident true 4-vertices at least

\begin{align*}
   \frac{d_{\gx}(f)-4+\rho^+(f)-\rho^-(f)-\frac{2}{3}s}{t}&\geq \frac{d_{\gx}(f)-4-\frac{2}{3}s}{t}\\
  &\geq\frac{d_{\gx}(f)-4+\frac{2}{3}t-\frac{2}{3}\lfloor\frac{d_{\gx}(f)}{2}\rfloor}{t}\\
  &\geq\frac{d_{\gx}(f)-4+\frac{2}{3}t-\frac{2}{3}\cdot\frac{d_{\gx}(f)}{2}}{t}\\
  &=\frac{\frac{2}{3}d_{\gx}(f)-4+\frac{2}{3}t}{t},
\end{align*}
which is at least $\frac{2}{3}$ provided that $d_{\gx}(f)\geq 6$, and at least $\frac{1}{3}$  provided $d_{\gx}(f)=5$ and $t\geq 2$ (actually, if $d_{\gx}(f)=5$ then $t\leq 2$ since $s+t\leq\lfloor\frac{5}{2}\rfloor=2$).

If $d_{\gx}(f)=5$ and $t=1$, then by $s+t\leq2$, we have $s\leq 1$. So $f$ sends to its incident 4-vertex at least $5-4-\frac{2}{3}=\frac{1}{3}$ by R8 and Claim \ref{clm-new}.
\end{proof}









\begin{prop}
After the application of Rules, the charge of every face of $\gx$ is non-negative.
\end{prop}

\begin{proof}
 Claim \ref{clm-new} along with R7 and R8 deduce that $c'(f)\geq 0$ for each $4^+$-face $f$ and each $3$-face that is incident with a transitive false vertex.
Now we calculate the final charges of true 3-faces and  false 3-faces incident with one non-transitive false vertex.

First of all, we assume that $f_{1}=\{vv_{1}v_{2}\}$ is a true 3-face with $d_{\gx}(v)\leq d_{\gx}(v_{1})\leq d_{\gx}(v_{2})$.

 If  $v$ is a 3-vertex, then $v_{1}$ and $v_{2}$ are $24^{+}$-vertices, thus the charge of $f_1$ is at least $3-4+2\times\frac{5}{6}>0$ after applying R1--R6. Therefore, $c'(f)=0$ by R7.

 If  $v$ is a 4-vertex, then $v_{1}$ and $v_{2}$ are $12^{+}$-vertices, thus the charge of $f_1$ is at least $3-4+2\times\frac{2}{3}>0$ after applying R1--R6. Therefore, $c'(f)=0$ by R7.

 If $v$ is a 5-vertex, then $v_{1}$ and $v_{2}$ are $10^{+}$-vertices, thus $c'(f_{1})\geq3-4+2\times\frac{3}{5}>0$ by R5.

 If $v$ is a 6-vertex, then $v_{1}$ and $v_{2}$ are $9^{+}$-vertices, thus $c'(f_{1})\geq3-4+2\times\frac{5}{9}>0$ by R5.

 If $v$ is a $7^{+}$-vertex, then $v_{1}$ and $v_{2}$ are $8^{+}$-vertices, thus $c'(f_{1})\geq3-4+2\times\frac{1}{2}=0$ by R5.

On the other hand, assume that $f_{1}=\{vv_{1}v_{2}\}$ is a false face so that $v$ is a non-transitive false vertex and $d_{\gx}(v_{1})\leq d_{\gx}(v_{2})$.
Suppose that $v_{1}v_{3}$ crosses $v_{2}v_{4}$ in $G$ at $v$, and $f_{i}$ with $1\leq i \leq4$ is the face that is incident with $vv_{i}$ and $vv_{i+1}$ in $\gx$ (here $v_{5}$ is recognized as $v_{1}$).

If $d_{\gx}(v_{1})=3$, then $v_2,v_3$ are $24^+$-vertices. By R5, $v_2$ sends at least $\frac{5}{6}$ to $f_1$, and  by R6.1, $f_2$ sends at least $\frac{1}{6}$ to $f_1$. Therefore,
$c'(f_{1})\geq3-4+\frac{5}{6}+\frac{1}{6}=0$.

If $d_{\gx}(v_{1})=4$, then $v_2,v_3$ are $12^+$-vertices. If $f_{1}$ is a 4-special face, then $f$ receives $\frac{1}{6}$ from $v_{1}$ by R1, $\frac{2}{3}$ from $v_{2}$ by R5, at least $\frac{1}{6}$ from $f_{2}$ by R6.2, and thus $c'(f_{1})\geq3-4+\frac{2}{3}+\frac{1}{6}+\frac{1}{6}=0$.
If $f_{1}$ is not a 4-special face, then $f_{1}$ receives $\frac{2}{3}$ from $v_{2}$ by R5, $\frac{1}{3}$ from $f_{2}$ by R6.2, and thus $c'(f_{1})\geq3-4+\frac{2}{3}+\frac{1}{3}=0$.

If $d_{\gx}(v_{1})=5$, then $v_2,v_3$ are $10^+$-vertices. If $f_{1}$ is a 5-special face, then $f_{1}$ receives $\frac{3}{10}$ from $v_{1}$ by R2, $\frac{3}{5}$ from $v_{2}$ by R5, at least $\frac{1}{10}$ from $f_{2}$ by R6.3, and thus $c'(f_{1})\geq3-4+\frac{3}{10}+\frac{3}{5}+\frac{1}{10}=0$. If $f_1$ is not a 5-special face, then $f_{1}$ receives $\frac{1}{5}$ from $v_{1}$ by R2, $\frac{3}{5}$ from $v_{2}$ by R5, at least $\min\{\frac{1}{5},\frac{3}{10}\}=\frac{1}{5}$ from $f_{2}$ by R6.3, and thus $c'(f_{1})\geq3-4+\frac{1}{5}+\frac{3}{5}+\frac{1}{5}\geq0$.

If $d_{\gx}(v_{1})=6$, then $v_2,v_3$ are $9^+$-vertices. If $f_{1}$ is a 6-special face, then $f_{1}$ receives $\frac{7}{18}$ from $v_{1}$ by R3, $\frac{5}{9}$ from $v_{2}$ by R5, at least $\frac{1}{18}$ from $f_{2}$ by R6.4, and thus $c'(f_{1})\geq3-4+\frac{7}{18}+\frac{5}{9}+\frac{1}{18}=0$. If $f$ is not a 6-special face, then $f_{1}$ receives $\frac{1}{3}$ from $v_{1}$ by R3, $\frac{5}{9}$ from $v_{2}$ by R5, at least $\min\{\frac{1}{9},\frac{5}{18}\}=\frac{1}{9}$ from $f_{2}$ by R6.4, and thus $c'(f_{1})\geq3-4+\frac{1}{3}+\frac{5}{9}+\frac{1}{9}\geq0$.

If $d_{\gx}(v_{1})\geq 7$, then $v_{2}$ is a $8^{+}$-vertex. By R4 and R5, each of $v_1$ and $v_2$ sends at least $\frac{1}{2}$ to $f_1$, which implies $c'(f)\geq3-4+\frac{1}{2}+\frac{1}{2}=0$.
\end{proof}

For a true $k$-vertex $v$ of $\gx$, denote by $v_1,v_2,\cdots,v_k$  the neighbors of $v$ in $\gx$ that lie consecutively around $v$, and by $f_{i}$ the face that is incident with $vv_{i}$ and $vv_{i+1}$ in $\gx$ (the subscript is taken by modular $k$). These notations will be used in the proof of the next propositions without explaining their meanings again.

\begin{prop}
After the application of Rules, the charge of every 3-vertex of $\gx$ is non-negative.
\end{prop}

\begin{proof}

Suppose that $v$ is a 3-vertex of $\gx$.

(1) If $v$ is incident only with $3$-faces, then they are all true for otherwise $\gx$ has two adjacent false vertices or $G$ has a multi-edge. By Claim \ref{clm1}, $c'(v)\geq3-4+3\times\frac{2}{3}>0$.

(2) If $v$ is incident with one $4^{+}$-face, say $f_3$, and two 3-faces $f_1$ and $f_2$, then we consider the following subcases.

First, if $f_1$ and $f_2$ are true, then $c'(v)\geq3-4+2\times\frac{2}{3}>0$ by Claim \ref{clm1}.

Second, if only one of $f_1$ and $f_2$ is true, then by the symmetry, assume that $f_1$ is true. Therefore, $f_2$ is false and thus $v_3$ is a false vertex. If $f_3$ is a $5^+$-face, then it sends $\frac{2}{3}$ to $v$ by R8. By Claim \ref{clm1}, $f_1$ sends $\frac{2}{3}$ to $v$. Hence $c'(v)\geq 3-4+\frac{2}{3}+\frac{2}{3}>0$. On the other hand, if $f_3$ is a 4-face, then it is incident with only one false vertex $v_3$, and moreover, $v_3$ is not a transitive false vertex. Since $v_1$ is a $24^+$-vertex, $f_3$ send by R7 to $v$ at least $\frac{1}{2}\times (4-4+\frac{5}{6})=\frac{5}{12}$. Counting together the charge $\frac{2}{3}$ that $f_1$ sends to $v$ by  Claim \ref{clm1}, we conclude $c'(v)\geq 3-4+\frac{5}{12}+\frac{2}{3}>0$.

Third, if $f_1$ and $f_2$ are both false, then both $v_1$ and $v_3$ are false by Lemma \ref{n.adj}(c), and furthermore, $f_3$ is a $5^+$-face by  Lemma \ref{n.adj}(b), which sends $\frac{2}{3}$ to $v$ by R8. The face adjacent to $f_1$ in $\gx$ that is different from $f_2,f_3$ is denoted by $h_1$, and the face adjacent to $f_2$ in $\gx$ that is different from $f_1,f_3$ is denoted by $h_2$. By R6.1, each of $h_1$ and $h_2$ sends at least $\frac{1}{6}$ to $v$. Therefore, $c'(v)\geq3-4+\frac{2}{3}+2\times\frac{1}{6}\geq0$.

(3) If $v$ is incident with one $3$-face, say $f_1$, and two $4^+$-faces $f_2$ and $f_3$, then we consider the following subcases.

First, if $f_1$ is true, then it sends $\frac{2}{3}$ to $v$ by Claim \ref{clm1}. If $f_2$ or $f_3$, say $f_2$, is a $5^+$-face, then $v$ receives $\frac{2}{3}$ from $f_2$ by R8, which implies $c'(v)\geq3-4+\frac{2}{3}+\frac{2}{3}>0$. If $f_2$ and $f_3$ are both $4$-faces, then each of them is incident with at most one false vertex. Hence
by Claim \ref{new-2}(1), each of $f_2$ and $f_3$ sends at least $\frac{5}{12}$ to $v$, which implies $c'(v)\geq 3-4+\frac{2}{3}+2\times\frac{5}{12}>0$.

Second, if $f_1$ is false, then assume by the symmetry that $v_1$ is false. The face adjacent to $f_1$ in $\gx$ that is different from $f_2,f_3$ is denoted by $h_1$. By R6.1, $h_1$ sends at least $\frac{1}{6}$ to $v$.

If $v_3$ is true, then $f_3$ is either a 4-face that is not incident with two false vertices or a $5^+$-face, and so does $f_2$. By Claim \ref{new-2}(1) and R8, each of $f_2$ and $f_3$ sends at least $\min\{\frac{5}{12},\frac{2}{3}\}=\frac{5}{12}$ to $v$, which implies $c'(v)\geq 3-4+\frac{1}{6}+2\times\frac{5}{12}=0$.

If $v_3$ is false, then $f_2$ still sends at least $\frac{5}{12}$ to $v$ by the same reason as above. At this stage, if $f_3$ is a $5^+$-face, then it sends $\frac{2}{3}$ to $v$ by R8, and thus $c'(v)\geq 3-4+\frac{1}{6}+\frac{5}{12}+\frac{2}{3}>0$. Hence we assume that $f_3$ is a 4-face and let $f_3=\{vv_1u_3v_3\}$.

If $d_{\gx}(u_3)\geq 4$, then $h_1$ sends $\frac{1}{3}$ to $v$ by R6.1. If $f_2$ is a $5^+$-face now, then it sends $\frac{2}{3}$ to $v$ by R8, which implies $v'(v)\geq 3-4+\frac{1}{3}+\frac{2}{3}=0$. If $f_2$ is a $4$-face, then let $f_2=\{vv_2u_2v_3\}$. Since $u_2u_3\in E(G)$, $u_2$ or $u_3$ is a $8^+$-vertex. If $u_2$ is a $8^+$-vertex, then $f_2$ sends to $v$ at least $\frac{1}{2}+\frac{5}{6}=\frac{4}{3}$ by R5 \textcolor[rgb]{1.00,0.00,0.00}{and R7} and thus $c'(v)\geq 3-4+\frac{1}{3}+\frac{4}{3}>0$. If $u_3$ is a $8^+$-vertex, then $f_3$ sends to $v$ at least $\frac{1}{2}$ by R5. By Claim \ref{new-2}(1), $f_2$ sends to $v$ at least $\frac{5}{12}$. Therefore, $c'(v)\geq 3-4+\frac{1}{3}+\frac{1}{2}+\frac{5}{12}>0$. Note that $v_3$ is not a transitive false vertex on $f_2$, and neither $v_1$ nor $v_3$ is a transitive false vertex on $f_3$.

If $d_{\gx}(u_3)=3$, then we look at the degree of $f_2$. If $f_2$ is a $4$-face, then let $f_2=\{vv_2u_2v_3\}$. Since $u_2$ is adjacent to $u_3$ in $G$, $u_2$ is a $24^+$-vertex, which implies that $f_2$ sends to $v$ at least $2\times \frac{5}{6}=\frac{5}{3}$ by R5 and R7. Note that $v_3$ is not a transitive false vertex on $f_2$. Hence $c'(v)\geq 3-4+\frac{5}{3}>0$.  If $f_2$ is a $5^+$-face, then it gives $\frac{2}{3}$ to $v$ by R8. Suppose that the crossing $v_3$ is produced by $vw$ crossing $u_2u_3$ in $G$. Clearly, $w$ and $u_2$ are both $24^+$-vertices. Let $h_2$ be the face in $\gx$ that is incident with $wv_3$ and $u_2v_3$. By R6.1, $h_2$ sends $\frac{1}{6}$ to $v$. Recall that $h_1$ sends at least $\frac{1}{6}$ to $v$. We then have $c'(v)\geq 3-4+\frac{2}{3}+\frac{1}{6}+\frac{1}{6}=0$.

(4) If $v$ is incident only with $4^+$-faces, then we consider the following subcases.

If $v$ is incident with at least two $5^+$-faces, then it is clear that $c'(v)\geq 3-4+2\times\frac{2}{3}>0$ by R8.

If $v$ is incident with one $5^+$-face $f_3$ and two $4$-faces $f_1$ and $f_2$, then $f_3$ sends $\frac{2}{3}$ to $v$ by R8. If $f_1$ or $f_2$, say $f_1$, is incident with at most one false vertex, then by Claim \ref{new-2}(1), $f_1$ sends at least $\frac{5}{12}$ to $v$, which implies $c'(v)\geq 3-4+\frac{2}{3}+\frac{5}{12}>0$. Hence we assume that both $f_1$ and $f_2$ is incident with two false vertices. Let $f_{1}=\{vv_{1}u_{1}v_{2}\}$ and $f_{2}=\{vv_{2}u_{2}v_{3}\}$. We then conclude that $v_1,v_2,v_3$ are all false and $u_1u_2\in E(G)$. Therefore, $u_1$ or $u_2$ is a $8^+$-vertex, since any two $7^-$-vertices are not adjacent in $G$. By the symmetry, assume that $u_1$ is a $8^+$-vertex. By R5 and R7, $f_1$ sends at least $\frac{1}{2}$ to $v$, since neither  $v_1$ nor $v_2$ is a transitive false vertex on $f_1$. Hence $c'(v)\geq 3-4+\frac{2}{3}+\frac{1}{2}>0$.

If $v$ is incident only with $4$-faces, then let $f_{1}=\{vv_{1}u_{1}v_{2}\}$, $f_{2}=\{vv_{2}u_{2}v_{3}\}$ and $f_{3}=\{vv_{3}u_{3}v_{1}\}$. If there is at most one false vertex among $v_1,v_2$ and $v_3$, then each of $f_1,f_2$ and $f_3$ is incident with at most one false vertex. By Claim \ref{new-2}(1), each of them sends at least $\frac{5}{12}$ to $v$, which implies that $c'(v)\geq 3-4+3\times \frac{5}{12}>0$. If $v_1,v_2$ are false and $v_3$ is true, then $v_3$ is a $24^+$-vertex and $u_1u_2,u_1u_3\in E(G)$. If $d_{\gx}(u_1)\geq 8$, then $f_1$ sends at least $\frac{1}{2}$ to $v$ by R5 and R7, since neither $v_1$ nor $v_2$ is a transitive false vertex on $f_1$. Counting together the charge $2\times\frac{5}{12}=\frac{5}{6}$ receiving from $f_2$ and $f_3$ by Claim \ref{new-2}(1), we conclude $c'(v)\geq 3-4+\frac{1}{2}+\frac{5}{6}>0$. On the other hand, if $d_{\gx}(u_1)\leq 7$, then $d_{\gx}(u_2)\geq 8$. Since $v_2$ is not a transitive false vertex on $f_2$, $f_2$ sends at least $\frac{1}{2}+\frac{5}{6}=\frac{4}{3}$ to $v$ by R5 and R7, which immediately implies that $c'(v)\geq 3-4+\frac{4}{3}>0$. At last, we look at the case that $v_1,v_2,v_3$ are all false. This implies that $u_1u_2u_3$ is a triangle in $G$, and then at least two vertices among $u_1,u_2$ and $u_3$, say $u_1$ and $u_2$, are $8^+$-vertices. Since neither $v_1$ nor $v_2$ is a transitive false vertex on $f_1$, $f_1$ sends at least $\frac{1}{2}$ to $v$ by R5 and R7. Same result also holds for $f_2$. Therefore, $c'(f)\geq 3-4+\frac{1}{2}+\frac{1}{2}=0$.
\end{proof}

\begin{prop}
After the application of Rules, the charge of every 4-vertex of $\gx$ is non-negative.
\end{prop}

\begin{proof}
If $v$ is a false vertex, then it is clear that $c'(v)=c(v)=4-4=0$. Hence we assume in the following that $v$ is a true 4-vertex.
By Lemma \ref{n.adj}(d), $v$ is incident with at most three false 3-faces.

If $v$ is not incident with any $4$-special face, then $v$ sends out nothing and thus $c'(v)\geq c(v)=4-4=0$. So, we suppose that $f_1$ is a 4-special face so that $v_1$ is a false vertex. Let $w,u_4$ be vertices such that $vw$ crosses $v_2u_4$ in $G$ at $v_1$. Since $f_1$ is a 4-special face, $v_2u_{4}\in E(G)$ and $d_{\gx}(u_4)\leq 11$. This implies that $u_4\neq v_4$ (for otherwise $vv_4$ is an edge of type $(4,\leq 11)$), and thus $f_4$ is a $4^+$-face. Similarly, if $f_2$ is a 4-special face, then $f_3$ is a $4^+$-face. This implies that $v$ is incident with at most two 4-special faces, to which $v$ sends at most $2\times\frac{1}{6}=\frac{1}{3}$ by R1.

If $f_4$ is a $5^+$-face, or a 4-face incident with at most one false vertex, then by Claim \ref{clm3}, or Claim \ref{new-2}, $f_4$ sends at least $\frac{1}{3}$ to $v$. This implies $c'(v)\geq 4-4-\frac{1}{3}+\frac{1}{3}=0$.
If $f_4$ is a 4-face incident with exactly two false vertices, then $v_4$ is false. We now look at the face $f_2$.

If $f_2$ is not a false 3-face, then it is either a true 3-face, or a 4-face that is incident with at most one false vertex, or a $5^+$-face. In any case, $f_2$ sends
at least $\frac{1}{3}$ to $v$ by Claim \ref{clm2}, or Claim \ref{clm3}, or Claim \ref{new-2}. This concludes that $c'(v)\geq 4-4-\frac{1}{3}+\frac{1}{3}=0$. Hence we are left the case that $f_2$ is a false 3-face, that is, $v_2v_3\in E(\gx)$ and $v_3$ is false.

If $f_3$ is a $5^+$-face, then it sends at least $\frac{1}{3}$ to $v$ by Claim \ref{clm3}, which implies that $c'(v)\geq 4-4-\frac{1}{3}+\frac{1}{3}=0$. If $f_3$ is a 4-face, then let $f_3=\{vv_3u_3v_4\}$. Since $u_3u_4\in E(G)$, either $u_3$ or $u_4$ is a $8^+$-vertex. Without loss of generality, assume that $d_{\gx}(u_3)\geq 8$. Since neither $v_3$ nor $v_4$ is a transitive false vertex on $f_3$, $f_3$ sends at least $\frac{1}{2}$ to $v$ by R5 and R7. Therefore, $c'(v)\geq 4-4-\frac{1}{3}+\frac{1}{2}>0$.
\end{proof}

\begin{prop}
After the application of Rules, the charge of every 5-vertex of $\gx$ is non-negative.
\end{prop}

\begin{proof}
If a $5$-vertex $v$ is not incident with any 5-special face, then $c'(v)\geq 5-4-5\times \frac{1}{5}=0$ by R2. Hence we assume that $v$ is incident with at least one 5-special face.

Suppose that $f_1$ is a special 5-face so that $v_1$ is a false vertex. Let $w,u_5$ be vertices such that $vw$ crosses $v_2u_5$ in $G$ at $v_1$. Since $f_1$ is a 5-special face, $wv_2\in E(G)$ and $d_{\gx}(u_5)\leq 9$. This implies that $u_5\neq v_5$ (for otherwise $vv_5$ is an edge of type $(5,\leq 9)$), and thus $f_5$ is a $4^+$-face. This fact tells us that if $v$ is incident with a 5-special face, then it must be incident with one $4^+$-face. Hence $v$ is incident with at most four 3-faces, among which at most three are 5-special.

If $v$ is incident with three 5-special faces, then it is incident with two $4^+$-faces, and thus $c'(v)\geq 5-4-3\times\frac{3}{10}>0$ by R2. If $v$ is incident with at most two 5-special, then $c'(v)\geq 5-4-2\times\frac{3}{10}-2\times\frac{1}{5}=0$ by R2.
\end{proof}

\begin{prop}
After the application of Rules, the charge of every 6-vertex of $\gx$ is non-negative.
\end{prop}

\begin{proof}
The proof is highly similar to the previous one. For the completeness of the paper, we add this proof here.

If a $6$-vertex $v$ is not incident with any 6-special face, then $c'(v)\geq 6-4-6\times \frac{1}{3}=0$. Hence we assume that $v$ is incident with at least one 6-special face.

Suppose that $f_1$ is a special 6-face so that $v_1$ is a false vertex. Let $w,u_6$ be vertices such that $vw$ crosses $v_2u_6$ in $G$ at $v_1$. Since $f_1$ is a 6-special face, $v_2u_{6}\in E(G)$ and $d_{\gx}(u_6)\leq 8$. This implies that $u_6\neq v_6$ (for otherwise $vv_6$ is an edge of type $(6,\leq 8)$), and thus $f_6$ is a $4^+$-face. This fact tells us that if $v$ is incident with a 6-special face, then it must be incident with one $4^+$-face. Hence $v$ is incident with at most five 3-faces, among which at most four are 6-special.

If $v$ is incident with four 6-special faces, then it is incident with two $4^+$-faces, and thus $c'(v)\geq 6-4-4\times\frac{7}{18}>0$ by R3. If $v$ is incident with at most three 6-special, then $c'(v)\geq 6-4-3\times\frac{7}{18}-2\times\frac{1}{3}>0$ by R3.
\end{proof}

\begin{prop}
After the application of Rules, the charge of every $7^+$-vertex of $\gx$ is non-negative.
\end{prop}

\begin{proof}
If $v$ is a 7-vertex, then $v$ is incident with at most six false 3-faces, for otherwise two false vertices are adjacent in $\gx$. Hence we have, by R4, that $c'(v)\geq 7-4-6\times\frac{1}{2}=0$. If $v$ is a $8^+$-vertex, then $c'(v)\geq d_{\gx}(v)-4-\frac{d_{\gx}(v)-4}{d_{\gx}(v)}\cdot d_{\gx}(v)=0$ by R5.
\end{proof}

This is the end of the whole proof.
\end{proof}


\bibliographystyle{srtnumbered}
\bibliography{mybib}

\end{document}